\documentclass[12pt]{amsart}

\usepackage{amssymb,amsthm,amsmath,mathrsfs}
\usepackage{graphicx,float}
\usepackage{pst-all,pstricks}
\usepackage{amsfonts,dsfont}
\usepackage{url}
\usepackage{subfigure}
\input xy
\xyoption{all}
\usepackage{nicefrac}
\usepackage[T1]{fontenc}

\theoremstyle{theorem}
\newtheorem{theorem}{Theorem}[section]
\newtheorem*{theorem*}{Theorem}
\newtheorem{corollary}[theorem]{Corollary}
\newtheorem{lemma}[theorem]{Lemma}
\newtheorem{proposition}[theorem]{Proposition}
\theoremstyle{definition}
\newtheorem{definition}[theorem]{Definition}
\newtheorem{example}[theorem]{Example}
\newtheorem{remark}[theorem]{Remark}

\newcommand{\rank}{\operatorname{rk}}

\newcommand{\im}{\operatorname{im}}
\newcommand{\coker}{\operatorname{coker}}

\thanks{The author is partially supported by the Spanish Ministerio de Ciencia, Innovaci{\'{o}}n y Universidades (grant PGC2018-098321-B-I00).}

\begin{document}

\author{H. Barge}
\address{E.T.S. Ingenieros inform\'{a}ticos. Universidad Polit\'{e}cnica de Madrid. 28660 Madrid (Espa{\~{n}}a)}
\email{h.barge@upm.es}

\title{\v Cech cohomology, homoclinic trajectories and robustness of non-saddle sets}

\begin{abstract}
In this paper we study flows $\varphi:M\times\mathbb{R}\longrightarrow M$ having an isolated non-saddle set. We see that the complexity of the region of influence of an isolated non-saddle set $K$ depends on the way in which $K$ sits on the phase space at the cohomological level. We construct flows in surfaces having isolated non-saddle sets with a prescribed structure for its region of influence. We also study parametrized families of smooth flows and continuations of non-saddle sets. 
\end{abstract}

\maketitle

\section{Introduction and preliminaries}

In this paper we study the structure of flows $\varphi:M\times\mathbb{R}\longrightarrow M$ defined on locally compact ANR's having a connected isolated non-saddle set. The theory of non-saddle sets, first studied by Bhatia \cite{BhJDE} and Ura \cite{Ura}, although, according to Ura, introduced by Seibert in an oral communication, presents itself as a natural generalization of the classical theory of stability and attraction of the more recent theory of unstable attractors. 

Isolated non-saddle sets share many nice properties with the class of attractors (resp. repellers). From the topological point of view, as it happens with attractors, repellers and unstable attractors without external explosions, isolated non-saddle sets in ANR's have the shape of finite polyhedra. This property does not hold for general isolated invariant sets even in euclidean spaces. This was proven by    Giraldo, Mor\'on, Ruiz del Portal and Sanjurjo in \cite{GMRSo} and highlights that isolated non-saddle sets are a wide class  of isolated invariant which are suitable to be studied with topological  methods. On the other hand, from the dynamical perspective, isolated non-saddle sets present a simple local dynamical structure divided into pieces which are either attracted or repelled. The main differences with the aforementioned theories of attractors, repellers and unstable attractors without external explosions appear when looking to the global structure of these flows. More specifically, the complexity of a flow having an isolated non-saddle set relies in the structure of the region of influence of the non-saddle set which is especially rich in presence of the so-called dissonant points, a phenomenon that does not appear in the basin of attraction of an attractor neither stable nor unstable. The region of influence of an isolated non-saddle set has been deeply studied in \cite{BSdis, BSbif}. 

In spite of the similarities between the local dynamics of isolated non-saddle sets and attractors, while attractors are robust objects in both topological and dynamical senses, it is well-known that non-saddle sets are not. More specifically, small perturbations of a flow preserve attractors and their shape while small perturbations may transform isolated non-saddle sets into saddle ones and even modify their shape as it has been seen in \cite{GSRo}. However, it has been proven in \cite{BNLA, BSdis} that there are  some situations in which the robustness of the shape is equivalent to the robustness of non-saddleness.

The aim of this paper is to study some connections between the topological structure of the region of influence of an isolated non-saddle set and the way in which the non-saddle set lies in the phase space at the cohomological level. This is motivated from \cite[Theorem~18]{SGTrans} where it is proved that if $K$ is an isolated unstable attractor without external explosions in a manifold $M$, each homoclinic component of $\mathcal{A}(K)\setminus K$ induces a non-trivial, non-torsion cohomology class in $H^1(M;G)$ and these cohomology classes are independent.  A similar result for more general phase spaces is \cite[Corollary~4.8 ]{SGRMC}. In the case of  an isolated non-saddle set in an ANR $M$, the complement of  $K$ in its region of influence $\mathcal{I}(K)\setminus K$ may have components which contain homoclinic trajectories in addition to other kinds of trajectories. We shall see that each component $C$ of $\mathcal{I}(K)\setminus K$ containing homoclinic trajectories induces $k\geq 1$  non-torsion independent cohomology classes in $\check{H}^1(M;G)$, where $k$ is the number of different ways in which $K$ may be approached from $C$ minus $1$. Moreover, these cohomology classes belong to the kernel of the homomorphisms induced in \v Cech cohomology by the inclusion $i:K\hookrightarrow M$. Notice that $k=1$ in the case of a component comprised entirely of homoclinic trajectories. We also see that if $M$ is a compact $G$-orientable $n$-manifold, each component $C$ also induces $k$ independent non-torsion cohomology classes in $\check{H}^{n-1}(M;G)$ and $\check{H}^{n-1}(K;G)$ and that these cohomology classes are related by the homomorphism induced in \v Cech cohomology by the inclusion $i:K\hookrightarrow M$. These results allow us to give upper bounds to the complexity of the region of influence of isolated  non-saddle sets. Particularly interesting is the case in which the phase space is a closed orientable surface. In addition to all of this, motivated by the results \cite[Theorem~39]{BSdis} and \cite[Theorem~26]{BNLA} about the continuation properties of non-saddle sets we find necessary and sufficient conditions for the property of being non-saddle to be robust for families of differentiable flows defined smooth manifolds without further assumptions about the dimension or the cohomology of the manifold.  

In order to make the paper more readable we recall some basic concepts about topology and dynamical systems. 

\textbf{Manifolds.} We recall that an $n$-\emph{dimensional manifold}  $M$ is a second countable, Hausdorff topological space satisfying that each point has a neighborhood homeomorphic to $\mathbb{R}^n$.  On the other hand, a second countable Hausdorff space is said to be an $n$-\emph{manifold with boundary} if each point has either a neighborhood homeomorphic to $\mathbb{R}^n$ or to the upper half-space $\mathbb{H}^n=\{(x_1,\ldots,x_n)\in\mathbb{R}^n\mid x_n\geq 0\}$. If the aforementioned homeomorphisms can be chosen to be $C^\infty$ we shall say that $M$ is a \emph{smooth} or \emph{differentiable} manifold (resp. a smooth or differentiable manifold with boundary). For more information about manifolds we recommend to the reader the books by Lee \cite{Lee2000,Lee2013} and Milnor \cite{Mil}.

We are especially interested in 2-manifolds. Through this paper connected 2-manifolds will be called surfaces. We reccommend to the reader  the book \cite{Ma}  as a reference about the topology of surfaces.

\textbf{ANR's.} A metric space $X$ is said to be an \emph{Absolute neighborhood retract} or, shortly, an \emph{ANR} if it satisfies that whenever there exists an embedding $f:X\to Y$ of $X$ into a metric space $Y$ such that $f(X)$ is closed in $Y$, there exists a neighborhood $U$ of $f(X)$ such that $f(X)$ is a retract of $U$. Some examples of ANR's are manifolds, CW-complexes and polyhedra. Besides, an open subset of an ANR is an ANR and a retract of ANR is also an ANR. For more information about ANR's we recommend \cite{borre} and \cite{Hu}. 

\textbf{Algebraic Topology.} In this paper we shall use singular homology and cohomology, \v Cech cohomology and Alexander duality theorem. We shall use the notation $H_*(\cdot;G)$ and $H^*(\cdot;G)$ for singular homology and cohomology respectively and $\check{H}^*(\cdot;G)$ for \v Cech cohomology. The coefficients group $G$ is always assumed to be either $\mathbb{Z}$ or $\mathbb{Z}_2$. Since \v Cech and singular cohomology theories agree on ANR's we sometimes use both interchangeably. Some good references for this material are the book of Spanier \cite{Span}, Hatcher \cite{Hat} and Munkres \cite{Munk}. Some applications of homological techniques to dynamics can be found in the papers \cite{RSG,BGS, SGball}.

\textbf{Shape theory.} There is a form of homotopy which is very convenient to study  the global
{to\-po\-lo\-gi\-cal} properties of the invariant spaces involved in dynamics, namely
the \textit{shape theory} introduced and studied by Karol Borsuk.  Although we are not going to make a deep use of shape theory, we recommend to the reader the books \cite{Bormono, CP, DySe, Mar, MarSe} and the papers \cite{Hast, KapRod, SGRAC, SanMul, Sanjsta, sanjuni, Rob2, RobSal} to see some applications to dynamical systems. We shall use the fact that \v Cech cohomology is a shape invariant. 

The main references we follow for the basic concepts of dynamical are the books \cite{BhSz, Rob1, Pil, HSD}.

\textbf{Limit sets.} We recall that the \textit{omega}-limit and the \textit{negative omega}-limit of a
point $x\subset M$ are the sets 
\[
\omega (x)=\bigcap_{t>0}
\overline{x\cdot \lbrack t,\infty)},\quad \omega ^{\ast }(x)=\bigcap_{t>0}\overline{x\cdot (-\infty ,-t]}.
\]
 
\textbf{Sections and parallelizable flows.} Given a flow $\varphi:M\times\mathbb{R}\to M$ by a \emph{section} $S$, we mean a set which intersects each trajectory exactly in a point. 

The flow $\varphi$ is said to be \emph{parallelizable} if it admits a section $S$ such that the map $\sigma :M\to \mathbb{R}$ defined by the property $x\sigma (x)\in S$ is continuous. Notice that, if one section satisfies that condition, all of them do.

 If a flow is parallelizable and $S$ is a section, the map $h:S\times\mathbb{R}\to M$ defined by  $(x,t)\mapsto xt$ is a homeomorphism. A direct consequence of these considerations is that a section $S$ of a parallelizable flow is a strong deformation retract of $M$ and the deformation retraction is provided by the flow.

\textbf{Morse theory.} Given a smooth manifold $M$, a \emph{Morse function} on $M$ is a smooth map $f:M\longrightarrow\mathbb{R}$ whose critial points are non-degenerate, i.e., the Hessian matrix at those points is non-singular. A vector field $X$ on $M$ is said to be \emph{gradient-like} with respect to $f$ if the derivative of $f$ in the direction of $X$ is positive for every non-critical point. Notice that this means that $f$ is strictly increasing on the non-stationary trajectories of the local flow $\varphi_X$ induced by the vector field $X$. The structure of a gradient-like vector field is well-understood and has deep relationships with the topology of the manifold $M$. For more information about Morse theory we recommend to the reader the books \cite{Milnor1963, Matsumoto2002, Nicolaescu2011}.

\textbf{Invariant manifolds, stability, attractors and repellers.} 
The \emph{stable} and \emph{unstable manifolds} of an invariant compactum $K$ are defined respectively as the sets 
\[
W^s(K)=\{x\in M\mid \emptyset\neq\omega(x)\subset K\}, \quad W^u(K)=\{x\in M \mid \emptyset\neq\omega^*(x)\subset K\}.
\]

Through this paper an \emph{attractor} will be an asymptotically stable set while a \emph{repeller} will be a negatively asymptotically stable set. More precisely, an invariant compactum $K$ is said to be an attractor if it possesses a neighborhood  $U$ of $K$ such that $%
\emptyset\neq\omega (x)\subset K$ for every $x\in U$ and, in addition, every neighborhood $V$ of $K$
contains a neighborhood $W$ of $K$ such that $W\lbrack 0,\infty
)\subset V$. The latter condition is known as \emph{stability}. A repeller is an attractor for the reverse flow. In this case the stable manifold of $K$ turns out to be an open set and is called \emph{basin of attraction} of $K$ and denoted by $\mathcal{A}(K)$. In an analogous way, the unstable manifold of a repeller $K$ is an open set called \emph{basin of repulsion} of $K$ and denoted by $\mathcal{R}(K)$.

If $K$ is an attractor (resp. repeller), 
the restriction flow $\varphi |_{\mathcal{A}(K)\setminus K}$ (resp. $\varphi |_{\mathcal{R}(K)\setminus K}$) is parallelizable and its sections are compact. 

Although through this paper we require attractors to be stable, sometimes stability is dropped from the definition to consider a more general kind of attractors. We shall refer to those as \emph{unstable attractors}. For the reader interested in a detailed treatment of unstable attractors we recommend the papers \cite{MSGS, SGTrans, AthC, AthPac}.

\textbf{Isolated invariant sets and isolating blocks.} A compact invariant set $K$ is said to be an \emph{isolated invariant set} if it possesses a so-called \emph{isolating neighborhood}, that is, a compact neighborhood $N$ such that $K$ is the maximal invariant set in $N$, or
setting
\[
N^{+}=\{x\in N\mid x[0,+\infty )\subset N\},\quad N^{-}=\{x\in
N\mid x(-\infty ,0]\subset N\}; 
\]
such that $K=N^{+}\cap N^{-}$. Notice that $N^+$ and $N^-$ are compact and, respectively, positively and negatively invariant. For instance, attractors and repellers are isolated invariant sets.

To avoid trivial cases, when we consider an isolated invariant set, it will be implicit that it is a non-empty proper subset of the phase space unless otherwise specified.

We shall make use of a special type of {iso\-la\-ting} neighborhoods, the so-called \emph{isolating blocks}, which have good
topological properties. More precisely, an isolating block $N$ is an
isolating neighborhood such that there are compact sets $N^{i},N^{o}\subset
\partial N$, called the entrance and exit sets, satisfying
\begin{enumerate}
\item $\partial N=N^{i}\cup N^{o}$,

\item for every $x\in N^{i}$ there exists $\varepsilon >0$ such that $%
x[-\varepsilon ,0)\subset M\setminus N$

and for every $x\in N^{o}$ there exists $\delta >0$ such that $%
x(0,\delta ]\subset M\setminus N$,

\item for every $x\in \partial N\setminus N^{i}$ there exists $\varepsilon >0$ such that $%
x[-\varepsilon ,0)\subset \mathring{N}$

and for every $x\in \partial N\setminus N^{o}$ there exists $\delta >0$ such
that $x(0,\delta ]\subset \mathring{N}$.
\end{enumerate}

If the phase space is a smooth manifold and the flow is of class $C^r$ with $r\geq 1$, the isolating blocks can be chosen to be
manifolds with boundary which contain $N^{i}$ and $N^{o}$ as submanifolds
of their boundaries and such that $\partial N^{i}=\partial N^{o}=N^{i}\cap
N^{o}$. This kind of isolating blocks will be called \emph{isolating block manifolds}.

Associated to an isolating block $N$ there are defined two continuous functions
\[
t^o:N\setminus N^+\to[0,+\infty),\quad t^i:N\setminus N^-\to(-\infty,0]
\]
given by
\[
t^o(x):=\sup\{t\geq 0\mid x[0,t]\subset N\},\quad t^i(x):=\inf\{t\leq 0\mid x[t,0]\subset N\}.
\]
These functions are known as the \emph{exit time} and the \emph{entrance time} respectively.

Some useful references about isolated invariant sets and isolating blocks are \cite{Con, EaJDE, ConEast}. We also recommend the papers \cite{Mischaikow1995, Mischaikow1998, Mischaikow2001, Barge2020} which show some applications of the theory of isolated invariant sets to the study of the Lorenz equations.

The paper is structured as follows: in section~\ref{sec:2} we recall the basic notions of the theory of non-saddle sets, including that of the region of influence, and their fundamental properties. In section~\ref{sec:3}  we introduce the so-called \emph{complexity} of the region of influence of a connected isolated non-saddle set. This complexity is a number which encapsulates how complicated is the region of influence of the non-saddle set.  We see that this complexity has a strong relationship with the cohomology of the phase space (Theorem~\ref{homoclinic}) and, in particular, with the homomorphism induced in cohomology by the inclusion $i:K\hookrightarrow M$. As a consequence, we can infer interesting dynamical features only by looking at topological relationships. For instance, we see in Theorem~\ref{dyntor} that the complexity region of influence of a connected isolated non-saddle set in the $n$-dimensional torus is at most $1$. We also see in Theorem~\ref{dynsuf} that the region of influence of a connected isolated non-saddle sets in a closed orientable surface of genus $g$  has complexity at most the $g$ and that this bound is sharp. In addition, we construct flows on closed orientable surfaces which have connected isolated non-saddle sets whose regions of influence have complexity $g$ and satisfying some additional conditions.  Finally, in section~\ref{sec:4} we study robustness of non-saddle sets from the point of view of continuation theory. The main results of this section are Proposition~\ref{rchar} and Theorem~\ref{strongrob}. Both results establish  necessary and sufficient conditions for a connected isolated non-saddle set to be locally continued to a family of non-saddle sets. In particular,  Theorem~\ref{strongrob} establishes the equivalence between the  continuation of non-saddleness and the continuation of certain cohomological relations, i.e., the continuation of the dynamical property of non-saddleness turns out to be equivalent to the continuation of some properties of topological nature. 

\section{Isolated non-saddle sets and their region of influence}\label{sec:2}

We recall that an invariant compactum $K$ is said to be \emph{saddle} whether there exists a neighborhood $U$ of $K$ such that, for every neighborhood $V$ of $K$ there exists $x\in V$ such that the trajectory of $\gamma(x)$ leaves $U$ in the past and in the future, i.e., such that $x[0,+\infty)\nsubseteq U$ and $x(-\infty,0]\nsubseteq U$. Otherwise $K$ is said to be \emph{non-saddle}.

In this paper we are interested in those non-saddle sets which are isolated as invariant sets. Isolated non-saddle sets are characterized by the property of admitting isolating blocks of the form $N=N^+\cup N^-$ (see \cite{GMRSo}). Moreover, if $K$ is connected and $N$ is a connected isolating block, then $N$ is of the form $N=N^+\cup N^-$ \cite[Proposition~3]{BSdis}. Since the flows considered in this paper are defined on ANR's, these isolating blocks are also ANR's and the inclusion $i: K\hookrightarrow N $ is a shape equivalence \cite{GMRSo} and hence induces isomorphisms in \v Cech cohomology groups. Therefore, $K$ has the shape of a finite polyhedron and its \v Cech cohomology is of finite type, i.e. finitely generated in all dimensions and nonzero only for a finite number of them.

 The local dynamics near an isolated non-saddle set is rather simple. Each component of $N\setminus K$ is either attracted or repelled by $K$. In addition, the flow provides a deformation retraction from $N\setminus K$ onto $\partial N$. Notice that $\partial N$ is also an ANR and, hence, it has a finite number of components. In spite of the simplicity of the local structure of these flows, their global structure  may be far more complicated than the structure of a flow having either an attractor or an isolated unstable attractor without external explosions. This complexity is exhibitted in the structure of the region of influence of the non-saddle set $K$. The \emph{region of influence} of a non-saddle set $K$, is defined as the set 
\[
\mathcal{I}(K)=W^s(K)\cup W^u(K).
\]  
This set is an open subset of the phase space and its topological and the dynamical structures have been extensively studied in \cite{BSdis}. Although the topological and dynamical structures of $\mathcal{I}(K)$ share many features with those of the basin of attraction of an isolated attractor without external explosions the global structure of $\mathcal{I}(K)$ may be much more rich. As a matther of fact, while the flow restricted to the complement of an isolated stable or unstable attractor without external explosions is always parallelizable, this is not much the case for the flow restricted to $\mathcal{I}(K)\setminus K$ (see Figure~\ref{double}).

The region of influence $\mathcal{I}(K)$ is composed by three different kinds of points. 

\begin{enumerate}
\item \emph{Purely attracted points}, that is, points $x\in\mathcal{I}(K)$ with $\omega(x)\subset K$ and $\omega^*(x)\nsubseteq K$.

\item \emph{Purely repelled points}, that is, points $x\in\mathcal{I}(K)$ with $\omega^*(x)\subset K$ and $\omega(x)\nsubseteq K$.

\item \emph{Homoclinic points}, that is, points $x\in\mathcal{I}(K)$ with $\omega^*(x)\subset K$ and $\omega(x)\subset K$.
\end{enumerate}

We denote by $\mathcal{A}^*(K)$, $\mathcal{R}^*(K)$ and $\mathcal{H}(K)$ the sets of purely attracted, purely repelled and homoclinic points respectively. The three sets are invariant subsets of $M$ and they satisfy that $\mathcal{A}^*(K)\cup K$ and $\mathcal{R}^*(K)\cup K$ are closed and $\mathcal{H}(K)\setminus K$ is open (see \cite[Proposition~12]{BSdis}). The desired situation at this point would be that $\mathcal{I}(K)\setminus K$ was decomposed as a finite union of purely attracted components,  purely repelled components and homoclinic components. However, this only happens if and only if $\mathcal{H}(K)$ is a closed set which is not the case in general(see Figure~\ref{double}). What actually happens is that $\mathcal{I}(K)\setminus K$ decomposes as a finite union of purely attracted components, purely repelled components, purely homoclinic components and components that contain points of the three kinds \cite[Proposition~20]{BSdis}. Moreover, the flow restricted to the purely attracted, purely repelled and homoclinic components is parallelizable while the flow restricted to the other components is not. Notice that the components that contain the three kinds of points are exactly those that contain boundary points of $\mathcal{H}(K)$. 

We call \emph{dissonant points} to those points in $\partial \mathcal{H}(K)$ which are not in $K$. The previous discussion illustrates that all the interesting dynamical features in $\mathcal{I}(K)\setminus K$ occur in the components containing dissonant points. In fact, in absence of dissonant points, it has been seen in \cite[Theorem~19]{BSbif} the dynamics in $\mathcal{I}(K)\setminus K$ is qualitatively the same as the dynamics in the basin attraction of an isolated attractor without external explosions studied in \cite{SGTrans}.

\section{Dynamical complexity of the region of influence of non-saddle sets and the cohomology of the phase space}\label{sec:3}

In this section we study in what extent the topology of the phase space and the way in which the isolated non-saddle continuum $K$ sits on it at the cohomological level affects to the structure of the region of influence of $K$. Some results in this spirit were obtained by S\'anchez-Gabites \cite{SGTrans,SGRMC} in the case of isolated attractors without external explosions and by Barge and Sanjurjo \cite{BSdis} for isolated non-saddle continua in compact manifolds. Although the results we present here can be regarded as generalizations of the aforementioned results, they stress again that the structure of the region of influence of a non-saddle set is much more subtle than the region of attraction of an isolated attractor without external explosions.   

The following result, which generalizes \cite[Theorem~25]{BSdis} establishes cohomological obstructions to the existence of homoclinic trajectories and dissonant points in the region of influence of a non-saddle continuum for a flow defined on a locally compact ANR.

\begin{theorem}\label{structure}
Let $M$ be a connected, locally compact ANR and $K$ a connected isolated non-saddle set of a flow on $M$. Suppose that $H^{1}(M;G)=0$ or, more generally, that the homomorphism $i^*:\check{H}^1(M;G)\to\check{H}^1(K;G)$ induced by the inclusion is a monomorphism. Then, $K$ does not have dissonant points. Moreover, if $U$ is a component of $M\setminus K$, then the flow restricted to $U$ is either locally attracted by $K$ (i.e. all points lying in $U$ near $K$ are attracted by $K$) or locally repelled by $K$. Furthermore, if $N$ is an isolating block of $K$ of the form $N=N^+\cup N^-$ then each component of $M\setminus K$ contains exactly one component of  $\partial N$. 
\end{theorem}

\begin{proof}
Consider an isolating block $N$ of $K$ such that $N=N^{+}\cup N^{-}.$ Then,  the homomorphism $j^*:\check{H}^1(M;G)\to \check{H}^1(N;G)$ induced by the inclusion $j:N\hookrightarrow M$  is a monomorphism. This follows from the equality $i^*=k^*j^*$, where $k^*$ is the isomorphism induced in \v Cech cohomology by the inclusion $k:K\hookrightarrow N$. 

 Consider the initial part of the long exact  sequence of \v Cech cohomology the pair $(M,N)$
\begin{multline*}
0\to \check{H}^0(M,N;G)\to \check{H}^0(M;G)\to \check{H}^0(N;G)\to \check{H}^1(M,N;G)\\ 
\to \check{H}^1(M;G)\xrightarrow{j^*} \check{H}^1(N;G)\to  \ldots
\end{multline*}
 Since $M$ and $N$ are connected, the homomorphism $\check{H}^0(M;G)\to \check{H}^0(N;G)$ is an isomorphism and, since $j^*$ is a monomorphism, the exactness of the sequence ensures that $\check{H}^i(M,N;G)=0$ for $i=0,1$.

On the other hand,  by excision we get 
\[
\check{H}^i(M,N;G)\cong \check{H}^i(M\setminus K, N\setminus K;G)
\]
and, as a consequence, $\check{H}^i(M\setminus K, N\setminus K;G)=0$ for $i=0,1$. Taking this into account in the long exact sequence of \v Cech cohomology of the pair $(M\setminus K, N\setminus K)$ we get that the inclusion $N\setminus K\hookrightarrow M\setminus K$ induces an isomorphism between
$\check{H}^0(M\setminus K;G)$ and $\check{H}^0(N\setminus K;G)$. This proves that each component of $M\setminus K$ contains exactly one component of $N\setminus K$. Besides, since $N^+\cap N^-=K$ it easily follows that every component of $N\setminus K$ must be either contained in $N^+\setminus N^-$ or $N^-\setminus N^+$. This shows that each component of $M\setminus K$ is either locally attracted or locally repelled by $K$, which prevents $K$ of having homoclinic trajectories and, hence, dissonant points. The remaining part of the statement  follows easily from the fact that the flow provides a deformation retraction from $N\setminus K$ onto $\partial N$. 
\end{proof}

The previous result shows that the way in which $K$ sits on the phase space at the cohomological level  plays an important role in the way in which the components of $N\setminus K$, where $N$ is an isolated block of the form $N=N^+\cup N^-$, lie on $M\setminus K$. We would like to point out that the components of $N\setminus K$ keep the information about the different ways we have to approach $K$ from $\mathcal{I}(K)\setminus K$. To make this precise we define the following equivalence relation:

\bigskip

Let $N$ and $N'$ be two isolating blocks of the form $N^+\cup N^-$ of an isolated invariant continuum $K$ of a flow defined on an ANR. We say that two components $N_1$ of $N\setminus K$ and $N'_1$ of $N'\setminus K$ are related if there exists $T\in\mathbb{R}$ such that $N_1T\subset N'_1$.

\bigskip

\begin{definition}
We shall call $K$-end of $\mathcal{I}(K)\setminus K$ to the equivalence class of a component of $N\setminus K$ where $N$ is an isolating block of the form $N=N^+\cup N^-$.
\end{definition}

It is not difficult to see that given an isolating block $N=N^+\cup N^-$, every $K$-end of $\mathcal{I}(K)\setminus K$ is represented by one and only one component of $N\setminus K$ and, hence, there are only a finite number of them.

\begin{definition}
Let $M$ be a connected, locally compact ANR and suppose that $K$ is a connected isolated non-saddle set of a flow on $M$. Let $C_i$ be a component of $\mathcal{I}(K)\setminus K$.  We define the \emph{local complexity} of $C_i$ as the difference $k_i-1$ where $k_i$ denotes the number of $K$-ends of $\mathcal{I}(K)\setminus K$ which are contained in the component $C_i$.

The \emph{complexity} $\mathfrak{c}$ of $\mathcal{I}(K)$ is defined as the sum of the local complexities of its components or, equivalently, as the difference  $k-m$ where $k$ and $m$ denote the number of $K$-ends and the number of components of $\mathcal{I}(K)\setminus K$ respectively.
\end{definition}

\begin{remark}
From the previous definitions together with  the discussion about the topological structure of the region of influence of an isolated non-saddle it follows that:
\begin{itemize}

\item The (local) complexity is well defined since different representatives of a $K$-end are contained in the same component of $\mathcal{I}(K)\setminus K$.

\item A component $C_i$ of $\mathcal{I}(K)\setminus K$ does not have homoclinic trajectories if and only if its local complexity is zero. Therefore, $\mathcal{I}(K)\setminus K$ does not contain homoclinic points if and only if the complexity of $\mathcal{I}(K)$ is zero.

\item A component $C_i$ of $\mathcal{I}(K)$ with local complexity greater than $1$ has dissonant points.

\item The local complexity of a homoclinic component is $1$ and, as a consequence, in absence of dissonant points, the complexity of $\mathcal{I}(K)$ is exactly the number of homoclinic components of $\mathcal{I}(K)\setminus K$. In the general case, the complexity is an upper bound of the number of components of $\mathcal{I}(K)\setminus K$ which contain homoclinic trajectories.

\item Theorem~\ref{structure} ensures that the region of influence of a connected isolated non-saddle set $K$ of a flow on an ANR $M$ with $H^1(M;G)=0$ or, more generally, such that the homomorphism induced in 1-dimensional \v Cech cohomology by the inclusion $i:K\hookrightarrow M$ is a monomorphism, has complexity zero.
\end{itemize}
\end{remark}

The following two examples illustrate the fact that the local complexities carry more information than the complexity of $\mathcal{I}(K)$.

\begin{example}\label{example1}
Let $M$ be a closed orientable surface of genus $2$. We consider $M$ endowed with the flow $\varphi:M\times\mathbb{R}\longrightarrow M$ depicted in Figure~\ref{double2}. This flow has an isolated non-saddle continuum $K$ which is homeomorphic to a sphere with the interiors of four disjoint topological closed disks removed. The flow in $K$ is stationary and $\mathcal{I}(K)\setminus K$ has two connected components $C_1$ and $C_2$ both of them comprised entirely of homoclinic trajectories. As a consequence, both components $C_1$ and $C_2$ have local complexity $1$ and, then, the complexity of $\mathcal{I}(K)$ is $2$. Notice that $K$ does not have dissonant points.

\end{example}

\begin{figure}
\setlength{\unitlength}{1cm}
\begin{picture}(10,6.5)
\put(-.5,0){\includegraphics[scale=.4]{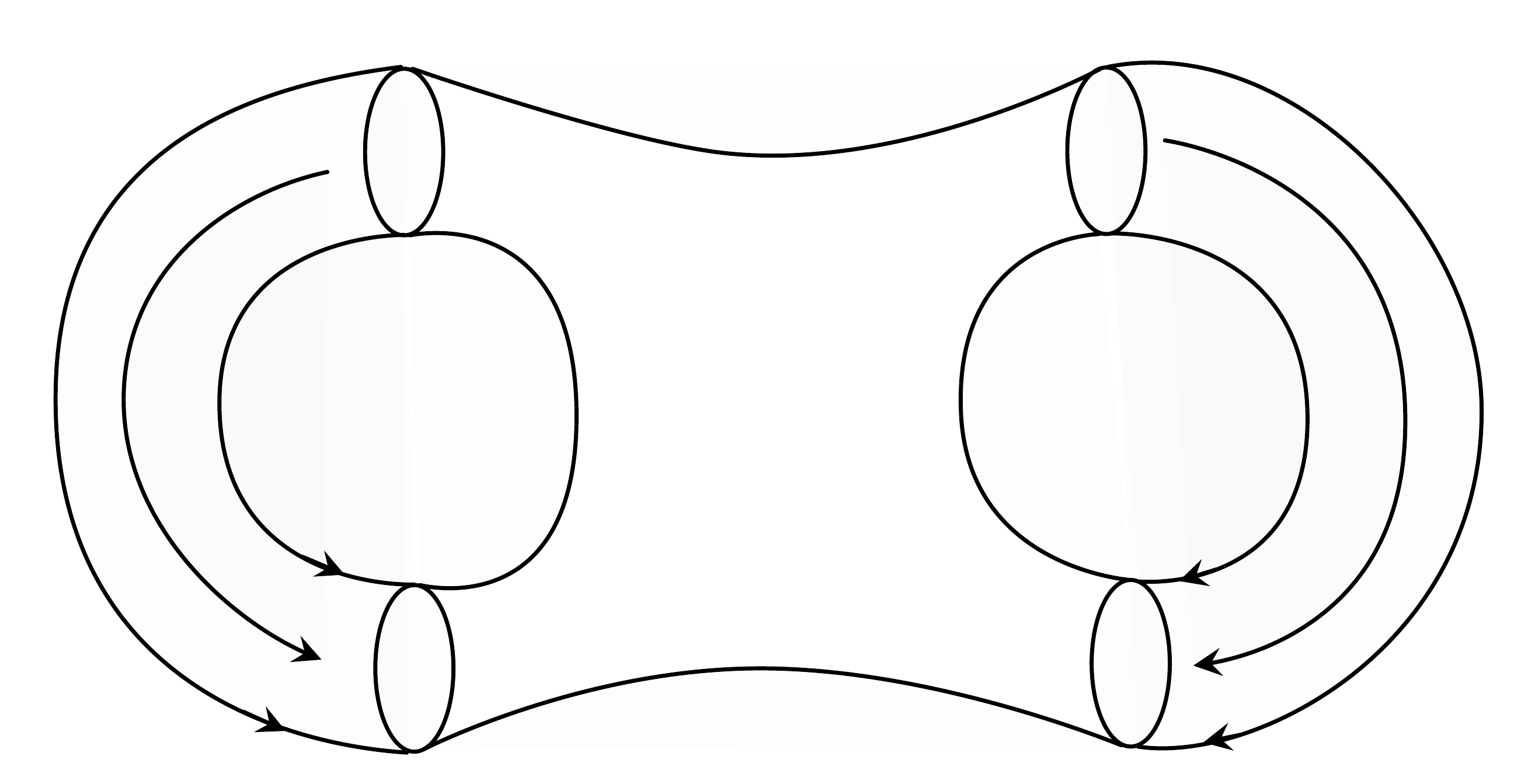}}
\put(4.5,2.5){$K$}
\put(1.2,2.5){$C_1$}
\put(8,2.5){$C_2$}
\end{picture}
\caption{Flow on a double torus which has an isolated non-saddle $K$ comprised of stationary points that is a sphere with the interiors of four disjoint topological closed disks removed. The region of influence of $K$ is the whole double torus and $\mathcal{I}(K)\setminus K$ has two components $C_1$ and $C_2$ with local complexity $1$.}
\label{double2}
\end{figure}
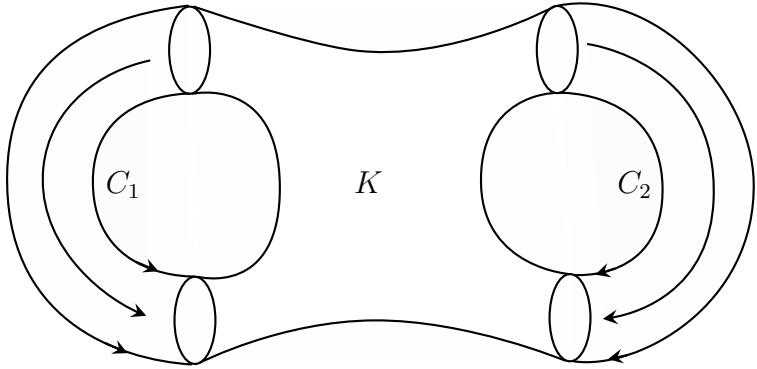

\begin{example}\label{example2}
Consider the flow $\hat{\varphi}:M\times\mathbb{R}\longrightarrow M$ defined on a closed orientable surface of genus $2$, depicted in Figure~\ref{double}. In this case, there is an isolated non-saddle continuum $K'$ which is homeomorphic to a sphere with the interiors of three disjoint topological closed disks removed. The flow $\hat{\varphi}$ is stationary in $K'$ and $\mathcal{I}(K')\setminus K'$ has two connected components $C_1'$ and $C_2'$. $C'_1$ is homeomorphic to an open annulus, every point in it is purely repelled by $K'$ and, hence, the local complexity of $C_1'$ is zero. On the other hand, $C_2'$ is homeomorphic to a 2-dimensional sphere with four punctures. The local complexity of $C_2'$ is $2$ and, as a consequence, it contains dissonant points. Observe that the dissonant points are those which lie in the stable and unstable manifolds of the fixed point $p\in\overline{C_2'}$.  It follows that the complexity of $\mathcal{I}(K)$ is $2$.
\end{example}

Examples \ref{example1} and \ref{example2} illustrate two flows defined on a closed orientable surface of genus $2$. Both of them have isolated non-saddle continua whose regions of influence have complexity $2$. In addition, both $\mathcal{I}(K)\setminus K$ and $\mathcal{I}(K')\setminus K'$ have two connected components. However, $K$ does not have dissonant points while $K'$ does. This stresses that the complexity does not predict in general the existence of dissonant points. On the other hand, if we look at the local complexities, we see that $\mathcal{I}(K')\setminus K'$ contains a component with local complexity $2$ and, hence, it must contain dissonant points. This discussion points out that the complexity can be seen as an upper bound of how complicated is the dynamics in the region of influence while the local complexities actually record the structure of the flow in $\mathcal{I}(K)\setminus K$.  In the 2-dimensional case, the existence of dissonant points can be detected using the Euler characteristic \cite[Theorem~32]{BSdis}.

The next result makes precise the relationship between the complexity of the region of influence of a connected isolated non-saddle set $K$ and the way in which $K$ sits on the phase space at the cohomological level. 

\begin{figure}
\setlength{\unitlength}{1cm}
\begin{picture}(10,8)
\put(-1.7,0){\includegraphics[scale=.4]{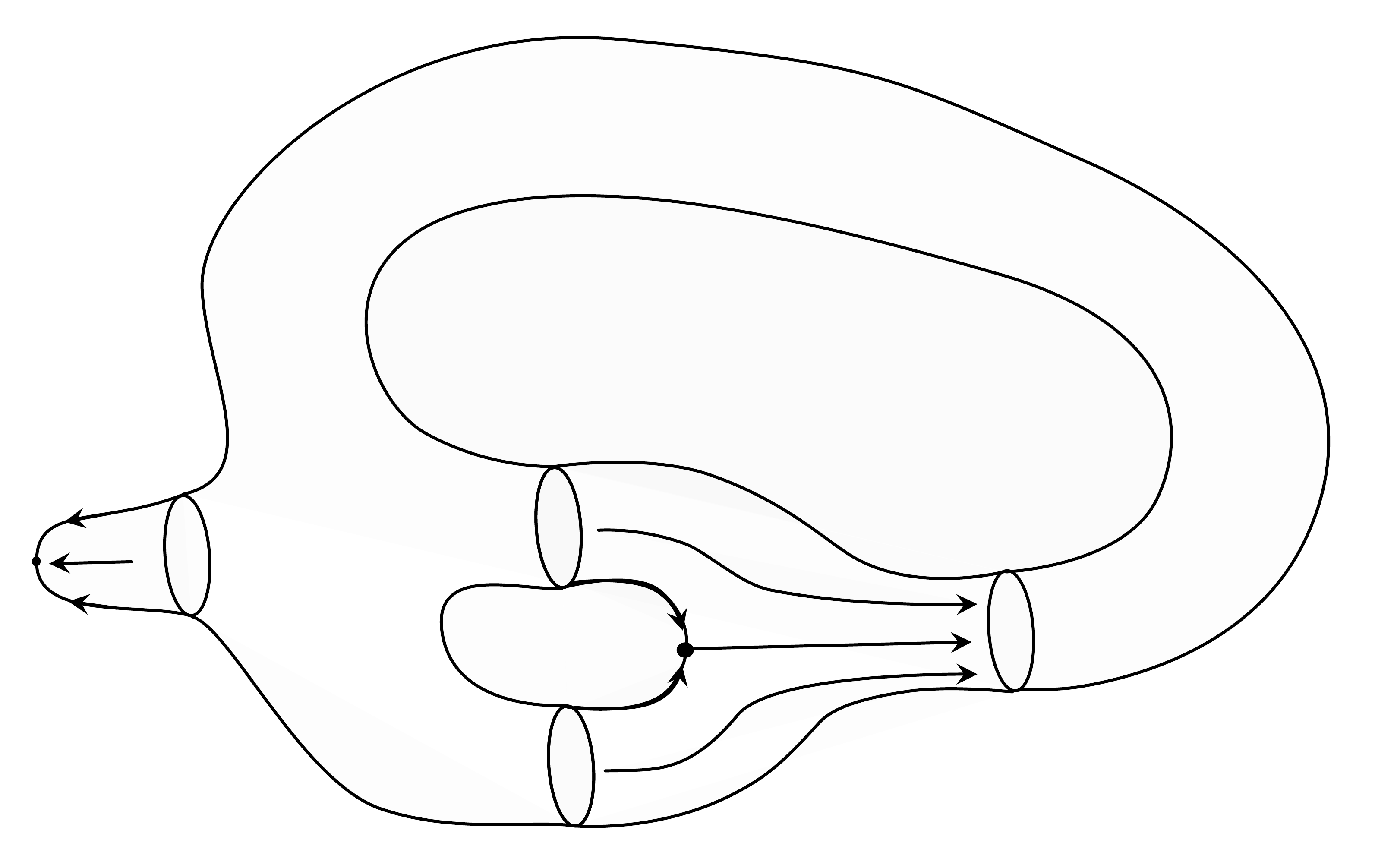}}
\put(1,2.8){$K'$}
\put(4.3,2){$p$}
\put(-.8,2){$C_1'$}
\put(5.6,.5){$C_2'$}
\end{picture}
\caption{Flow on a double torus having an isolated non-saddle set $K$ comprised of stationary points which is a sphere with the interiors of three closed topological disks removed. The region of influence of $K$ is the double torus with the point $p$ removed. $\mathcal{I}(K)\setminus K$ has only one connected component that has local complexity $2$.\label{double}}
\end{figure}
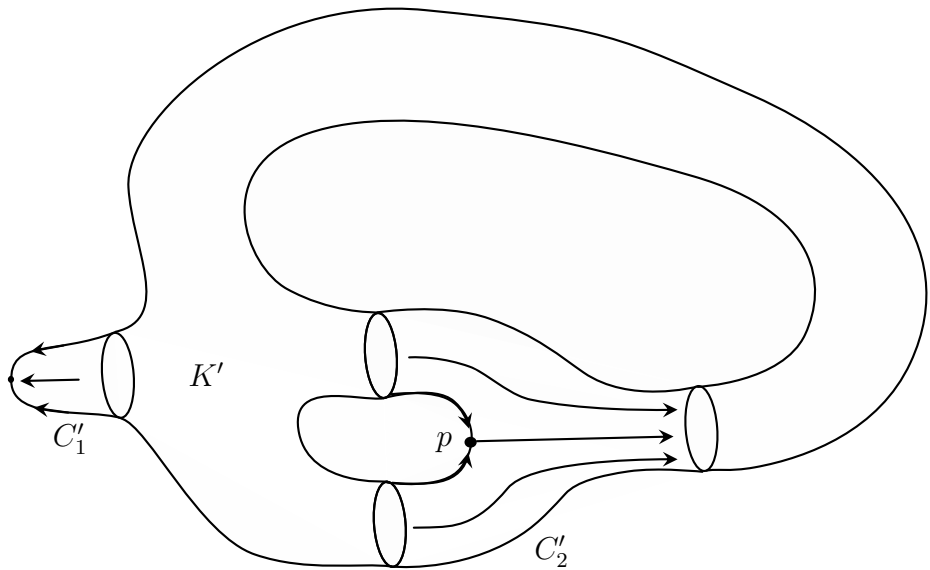

\begin{theorem}\label{homoclinic}
Let $K$ be an isolated non-saddle continuum of a flow defined on a connected, locally compact ANR $M$ and let $i:K\hookrightarrow M$ be the inclusion map. Suppose that the complexity of the region of influence of $K$ is $\mathfrak{c}$. Then, there exist  
\[
\alpha_1,\ldots,\alpha_{\mathfrak{c}}\in \check{H}^1(M;G)
\] 
which are independent non-torsion elements satisfying that $i^*(\alpha_i)=0$ for every $i=1,\ldots,\mathfrak{c}$. Moreover, if $M$ is a closed, connected and $G$-orientable $n$-manifold, then there exist  
\[
\beta_1,\ldots,\beta_{\mathfrak{c}}\in \check{H}^{n-1}(M;G)\quad\mbox{and}\quad \gamma_1,\ldots,\gamma_{\mathfrak{c}}\in \check{H}^{n-1}(K;G)
\]
which are independent non-torsion elements  such that $i^*(\beta_i)=\gamma_i$ for each $i=1,\ldots,\mathfrak{c}$.
\end{theorem}

\begin{proof} 

Let $N=N^+\cup N^-$ be an isolating block of $K$.  Reasoning as in the proof of Theorem~\ref{structure} it follows  that the homomorphism $j^*:\check{H}^1(M;G)\to \check{H}^1(N;G)$ induced by the inclusion $j:N\hookrightarrow M$  satisfies that $\ker j^*=\ker i^*$.  Consider the initial part of the long exact sequence of \v Cech cohomology of the pair $(M,N)$,
\begin{multline*}
0\to \check{H}^0(M,N;G)\to\check{H}^0(M;G)\to\check{H}^0(N;G)\to\check{H}^1(M,N;G)\\
\to\check{H}^1(M;G)\xrightarrow{j^*}\check{H}^1(N;G)\to\cdots 
\end{multline*}
Since $M$ and $N$ are connected, the second homomorphism is an isomorphism and, hence, $\check{H}^0(M,N)=0$ and $\check{H}^1(M,N)\cong\ker i^*$. Then, by excising $K$, we obtain that $\check{H}^0(M\setminus K,N\setminus K;G)=0$ and $\check{H}^1(M\setminus K,N\setminus K;G)\cong\ker i^*$. Therefore, the initial part of the long exact sequence of \v Cech cohomology of the pair $(M\setminus K,N\setminus K)$ takes the form 
\[
 0\to \check{H}^0(M\setminus K;G)\to \check{H}^0(N\setminus K;G)\to\ker i^*\to\cdots
\]

Let $C_1,\ldots, C_k$ be the components of $M\setminus K$. The exactness of the latter sequence ensures that $\ker i^*$ has a subgroup isomorphic to $H_1\oplus\cdots\oplus H_k$ where
\[
H_i=G\oplus\overset{m_i-1}\cdots\oplus G
\]
and $m_i$ is the number of components of $N\setminus K$ contained in $C_i$. Then, the first part of the result follows by observing that given a component $C_i$ of $M\setminus K$, each component $U_j$ of $\mathcal{I}(K)\setminus K$ contained in $C_i$ contributes with at least $l_j$ summands $G$ to $H_i$, where $l_j$ is the local complexity of $U_j$.

\bigskip

Let us prove the second part of the statement. Observe that, by Alexander duality, $\check{H}^n(K;G)\cong H_0(M,M\setminus K;G)=0$ and consider the terminal part of the long exact sequence of \v Cech cohomology of the pair $(M,K)$,
\[
\cdots\to \check{H}^{n-1}(M;G)\to\check{H}^{n-1}(K;G)\to\check{H}^n(M,K;G)\to \check{H}^n(M;G)\to \check{H}^n(K;G)=0.
\]
This long exact sequence breaks into the short exact sequence
\[
0\to\coker i^*\to\check{H}^n(M,K;G)\to \check{H}^n(M;G)\to 0,
\]
and, hence, $\check{H}^n(M,K;G)\cong\coker i^*\oplus G$. Another application of Alexander duality theorem ensures that $H_0(M\setminus K;G)\cong \coker i^*\oplus G$.

On the other hand, if $N=N^+\cup N^-$ by Alexander duality we get $H_1(\mathring{N},\mathring{N}\setminus K;G)\cong \check{H}^{n-1}(K;G)$ and
\[
\rank\check{H}^{n-1}(K;G)=\rank\coker i^*+\rank\im i^*=\rank\widetilde{H}_0(M\setminus K;G)+\rank\im i^*
\]

Taking this into account in the long exact sequence of reduced singular homology of the pair of ANR's $(\mathring{N},\mathring{N}\setminus K)$, it follows that 

\[
\rank \widetilde{H}_0(\mathring{N}\setminus K;G)\leq \rank\widetilde{H}_0(M\setminus K;G)+\rank\im i^*.
\]
 
Since the complexity of $\mathfrak{c}$ of $\mathcal{I}(K)$ satisfies that 
\[
\mathfrak{c}=\rank\widetilde{H}_0(\mathring{N}\setminus K;G)-\rank\widetilde{H}_0(\mathcal{I}(K)\setminus K;G)\leq\rank\widetilde{H}_0(\mathring{N}\setminus K;G)-\rank\widetilde{H}_0(M\setminus K;G),
\]
 it follows that there must be $\gamma_1,\ldots,\gamma_\mathfrak{c}$ independent non-torsion cohomology classes in $\im i^*$ and the result follows.

\end{proof}

A direct consequence of Theorem~\ref{homoclinic} is the following result which generalizes \cite[Theorem~4.6]{SGRMC}.

\begin{corollary}\label{homocoro}
Suppose that $M$ is an ANR and $K$ is an isolated {non-\-sad\-dle} continuum. Then, the complexity of $\mathcal{I}(K)$ is zero if and only if the {homomor\-phism}  $i^*:\check{H}^1(\mathcal{I}(K);G)\to \check{H}^1(K;G)$, induced by the inclusion, is injective. 
\end{corollary}

 The following result extends to the case of non-saddle sets \cite[Example~24]{SGTrans}.

\begin{proposition}\label{dyntor}
Suppose $K$ is an isolated non-saddle continuum in the $n$-{di\-men\-sio\-nal} torus $T^n$. Then, the complexity of $\mathcal{I}(K)$ is at most $1$. 
\end{proposition}
\begin{proof}
We shall assume that $n>1$ since for $n=1$ the result follows directly from Theorem~\ref{homoclinic}. To prove the result we use the fact that the cohomology ring $\check{H}^*(T^n;G)$ is the exterior algebra with $n$ generators $\omega_1,\ldots,\omega_n\in \check{H}^1(T^n;G)$. In particular, we use this ring structure to show that if $\rank(\ker i_{1}^*)>1$, then $i^*_{n-1}=0$. Here we denote by $i^*_k$ the homomorphism induced by the inclusion $i:K\hookrightarrow M$ in $k$-dimensional \v Cech cohomology.   Suppose that there exist $\alpha_1,\alpha_2\in \ker i^*_1$, with $\alpha_1$ and $\alpha_2$ linearly independent. Let $\{\alpha_1,\alpha_2,\ldots,\alpha_n\}$ be a basis of $\check{H}^1(T^n;G)$ containing $\alpha_1,\alpha_2$. The structure of the cohomology ring $\check{H}^*(T^n;G)$ ensures that any element $\beta\in \check{H}^{n-1}(T^n;G)$ is of the form $\sum_{i=1}^n m_i(\alpha_1\smile\ldots\smile\widehat{\alpha}_i\smile\ldots\smile \alpha_{n})$, where the \emph{hat} symbol $\;\widehat{}\;$ over $\alpha_i$ denotes that this cohomology class is removed from the cup product. Then, since
\[
i_{n-1}^*(\beta)=\sum_{i=1}^n m_i(i_1^*(\alpha_1)\smile\ldots\smile \widehat{i_1^*(\alpha_i)}\smile\ldots\smile i_1^*(\alpha_{n}))
\]
and each summand must contain either $i_1^*(\alpha_1)$ or $i_1^*(\alpha_2)$, it follows that $i^*_{n-1}(\beta)=0$. Therefore $i^*_{n-1}$ is the zero homomorphism and the result follows from Theorem~\ref{homoclinic}.
\end{proof}

The last part of this section deals with the case of isolated non-saddle continua in closed surfaces. In this context, Theorem~\ref{homoclinic} allows us to get sharp estimates on the complexity of the region of influence of $K$. In addition, Theorem~\ref{structure} allows us to ensure that some isolated non-saddle continua must be either attractors or repellers.

\begin{proposition}\label{nonsep}
Let $K$ be an isolated non-saddle continuum of a flow on a closed surface $M$. If $\rank\check{H}^1(K;\mathbb{Z}_2)=\rank H^1(M;\mathbb{Z}_2)$ and $K$ does not disconnect $M$, then the complexity of $\mathcal{I}(K)$ must be zero. Moreover, $K$ must be either an attractor or a repeller. 
\end{proposition}
\begin{proof}
 Since $K$ is a {non-se\-pa\-ra\-ting} continuum, Alexander duality ensures that
\[
 \check{H}^2(M,K;\mathbb{Z}_2)\cong H_0(M\setminus K;\mathbb{Z}_2)\cong\mathbb{Z}_2.
 \] 
 
Let us consider the long exact sequence of reduced \v Cech cohomology of the pair $(M,K)$,
\[
0\to \check{H}^1(M,K;\mathbb{Z}_2)\to\check{H}^1(M;\mathbb{Z}_2)\to\check{H}^1(K;\mathbb{Z}_2)\to \check{H}^2(M,K;\mathbb{Z}_2)\to\check{H}^2(M;\mathbb{Z}_2)\to 0.
\] 
The previous observation guarantees that the last homomorphism must be an isomorphism. As a consequence, the homomorphism $i^*:\check{H}^1(M;\mathbb{Z}_2)\to \check{H}^1(K;\mathbb{Z}_2)$ is surjective and, since $\rank \check{H}^1(K;\mathbb{Z}_2)=\rank\check{H}^1(M;\mathbb{Z}_2)$, it must be an isomorphism. Therefore, it follows from Theorem~\ref{structure} that the connected set $M\setminus K$ is either locally attracted or locally repelled by $K$ and, hence, $K$ is either an attractor or a repeller.

\end{proof}

Before stating the last result of the section we need to introduce some definitions.

\begin{definition}
We shall say that a fixed point $p$ of a flow $\varphi:M\times\mathbb{R}\longrightarrow M$ defined on a surface is \emph{topologically hyperbolic} if if possesses a neighborhood $U_p$ such that the flow in $U_p$ is topologically equivalent to the flow on $\mathbb{R}^2$ induced by the vector field $X(x,y)=(\lambda x,\mu y)$, where $\lambda$ and $\mu$ are either $+1$ or $-1$. We shall say that $p$ is a \emph{topologically hyperbolic saddle} whether $\lambda=1$ and $\mu=-1$ or viceversa. 
\end{definition}

\begin{definition}
Under the same assumptions, we shall say that $p$ is a \emph{degenerate saddle} if it possesses a neighborhood $U_p$ such that the flow in $U_p$ is topologically equivalent to a flow in $\mathbb{R}^2$ generated by a vector field of the form $X(x,y)=(\rho(x,y),0)$ where $\rho:\mathbb{R}^2\longrightarrow[0,+\infty)$ is a  non-negative smooth function which takes the value $0$ only at $(0,0)$ (see Figure~\ref{degenerate}).
\end{definition}

\begin{figure}
\includegraphics[scale=0.4]{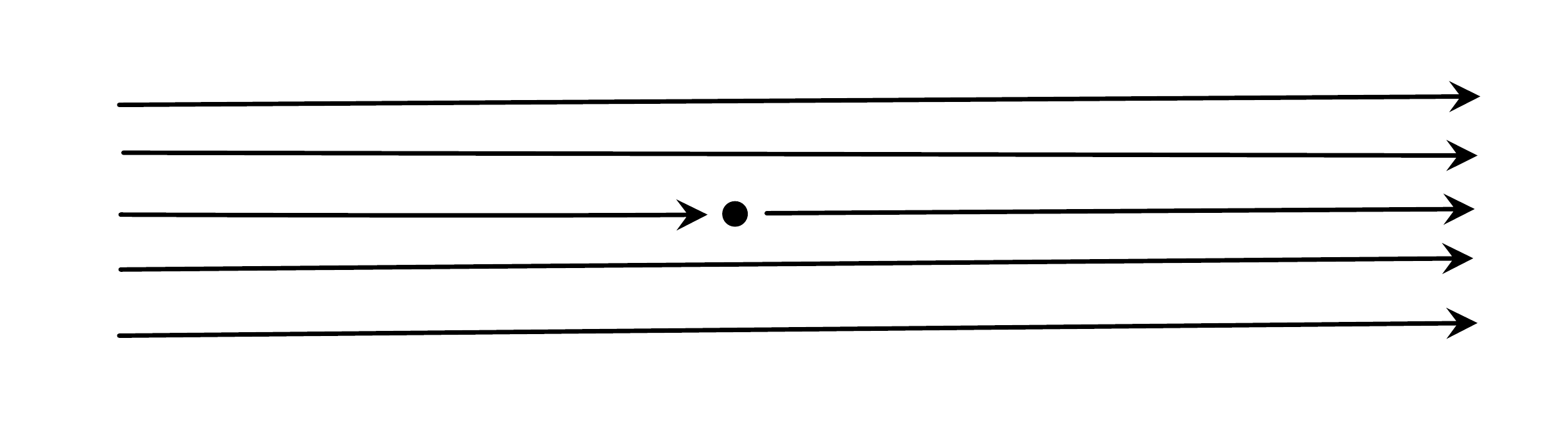}
\label{degenerate}
\caption{Model for a degenerate saddle fixed point.}
\end{figure}

\begin{theorem}\label{dynsuf}
Assume that $K$ is a connected isolated non-saddle set of a flow defined on a closed, orientable surface $M$ of genus $g$. Then,
\begin{enumerate}
\item[(i)] the complexity of $\mathcal{I}(K)$ is at most $g$.

\item[(ii)] Given $k_1,\ldots, k_n$ non-negative integers such that $g=k_1+\ldots+ k_n$ there exists a  flow on $M$ having an isolated non-saddle continuum $K$ whose region of influence $\mathcal{I}(K)$ has complexity $g$ and sastisfying that $\mathcal{I}(K)\setminus K$  has $n$ components $C_1,\ldots, C_n$ with local complexities $k_1,\ldots, k_n$ respectively. Moreover, the flow can be constructed with the following properties:
\begin{enumerate}
\item Every component $C_i$ with non-zero local complexity has dissonant points.

\item If we denote by $k$, $m$ and $l$ the number of components $C_i$ with complexities greater, equal, and less than $1$ respectively,  $M\setminus \mathcal{I}(K)$ is comprised of $g-k+l$ isolated fixed points where $g-k-m$ are topologically hyperbolic saddles, $m$ are degenerate saddle fixed points and $l$ are attracting fixed points.
\end{enumerate} 
\end{enumerate}
\end{theorem}

\begin{proof}
Since $K$ is proper subcontinuum of a surface \cite[Corolario~B.9]{SGth} ensures that $\check{H}^1(K;G)=G\oplus\overset{r}\cdots\oplus G$. Then,  $\check{H}^1(M;G)=\ker i^*\oplus\im i^*$ and using the fact that  $\rank \check{H}^1(M;G)=2g$  we get that either the rank of $\ker i^*$ or the rank of $\im i^*$ is at most $g$.  As a consequence, Theorem~\ref{homoclinic} ensures (i). 

Consider non-negative integers $k_1,\ldots, k_n$ such that $k_1+\ldots+k_n=g$ and suppose that $k_i=0$ for every $i\leq l$ and $k_i\geq 1$ if $i\geq l+1$. To prove (ii) we first observe that a closed orientable surface of genus $g$ can be constructed gluing together the 2-manifolds with boundary $K$, $D_1,\ldots, D_l$ and $H_{l+1},\ldots, H_n$ where :

\begin{enumerate}
\item $K$ is a sphere  with the interiors of $g+n$ disjoint topological closed disks removed.

\item  Each $D_i$ is a topological closed disk.

\item Each $H_i$ is a sphere with the interiors of $k_i+1$ disjoint topological closed disks removed.
\end{enumerate}

We obtain a closed surface $M$ by attaching to $K$ the $D_i$'s and the $H_i$'s in such a way that each $D_i$ caps a hole of $K$ and each $H_i$ connects $k_i+1$ holes of $K$. This can be done in such a way that the surface $M$ obtained is orientable. The Euler characteristic of this surface is
\begin{multline*}
\chi(M)=\chi(K)+\sum_{i=1}^l\chi(D_i)+\sum_{i=l+1}^n\chi(H_i)=(1-g-n)+l+\sum_{i=l+1}^n(1-k_i)\\
=(1-g-n)+l+(n-l+1)-g=2-2g
\end{multline*}
and, as a consequence, $M$ is a closed orientable surface of genus $g$.

We shall define a flow $\varphi:M\times\mathbb{R}\longrightarrow M$ with the desired properties.  We assume that $\varphi$ is stationary in $K$. As a consequence, the $D_i$'s and the $H_i's$ are also invariant and the flow is stationary in their boundaries.

 Let us define the flow on the disks $D_i$. Each disk $D_i$ comes equipped with a homeomorphism $h_i:D_i\longrightarrow D$ to the closed unit disk $D\subset\mathbb{R}^2$. Consider in $\mathbb{R}^2$ the flow $\varphi_{X}$ induced by the vector field $X(x,y)=(-x,-y)$. This flow has the origin as a global attractor. Since the unit circle is a global section of the flow $\varphi_{X}|_{\mathbb{R}^2\setminus\{(0,0)\}}$ we can use Beck's theorem \cite{Beck} to construct a new flow $\bar{\varphi}$ in $\mathbb{R}^2$ leaving fixed every point of the unit circle and breaking every trajectory $\gamma$ of the original flow, different from the origin, into three trajectories: the oriented ray of $\gamma$ which connects infinity with a point of $\partial D$, the endpoint of this ray, and the oriented ray which connects this point with the origin. We observe that the closed unit disk $D$ is invariant under $\bar{\varphi}$ and that the pair $(\{0\},\partial D)$ is an attractor-repeller decomposition of $\bar{\varphi}|_D$ (see Figure~\ref{disk}). We define $\varphi|_{D_i}=h_i^{-1}\circ\bar{\varphi}|_D$.

\begin{figure}
\centering
\includegraphics[scale=.4]{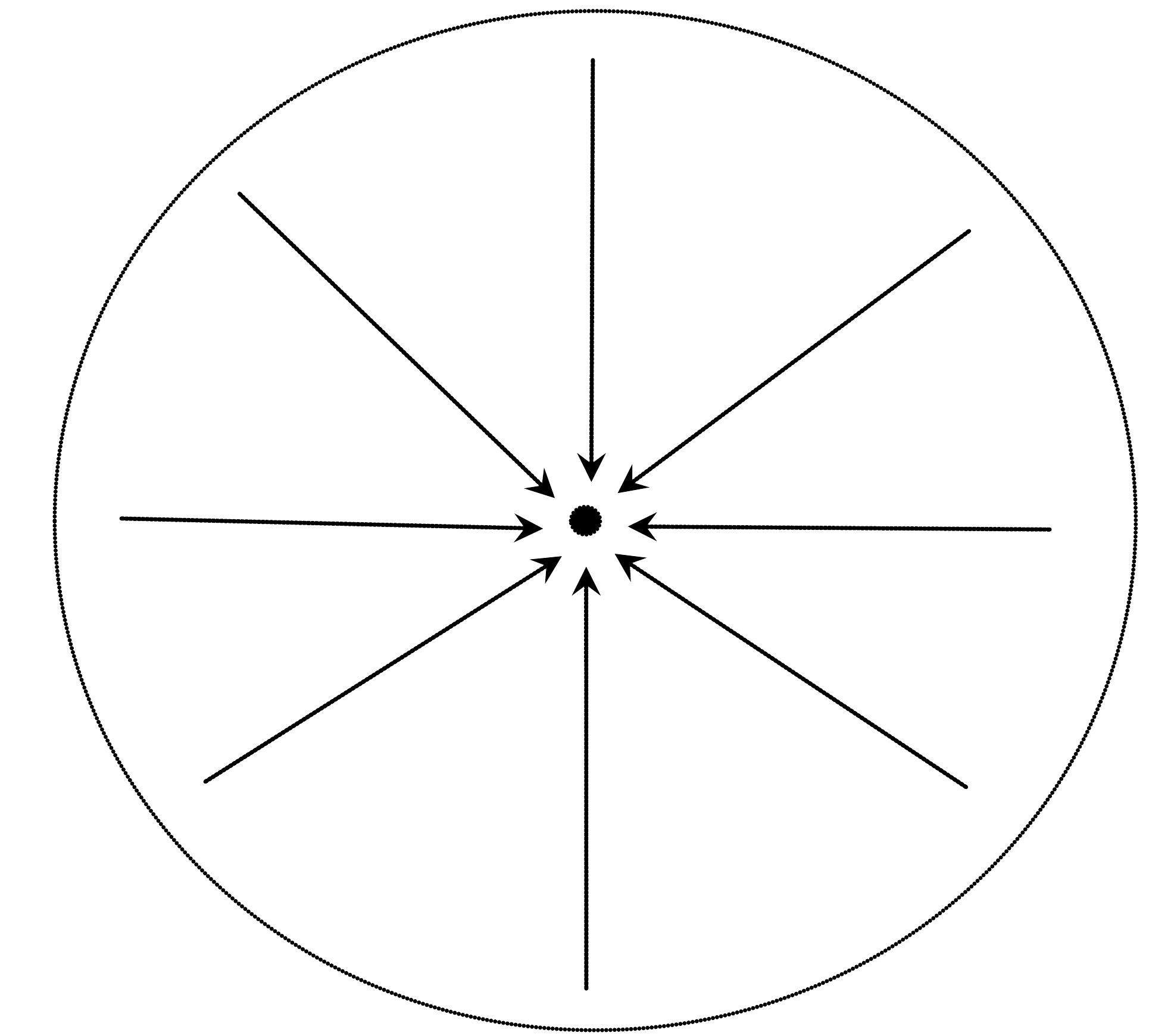}
\caption{Flow on a disk which has its center $C$  as an asymptotically stable fixed point and its boundary as a repelling circle of fixed points. The pair $(\{0\},\partial D)$ is an attractor-repeller decomposition of this flow.}
\label{disk}
\end{figure}

To define the flow in the $H_i$'s we have to separate two different situations: (A) $k_i=1$ and (B) $k_i>1$.  To construct te flow in case (A) observe that if $k_i=1$ there exists a homeomorphism $\bar{h}_i: H_i\longrightarrow S^1\times [0,1]$. We fix a point $z\in S^1$ and consider in $S^1\times [0,1]$ the flow $\phi$ which is stationary in the boundary $S^1\times\{0,1\}$ and such that the trajectories of points in $(S^1\times(0,1))\setminus ({z}\times(0,1))$ move from $S^1\times\{0\}$ to $S^1\times\{1\}$ along the fibers while the fiber $\{z\}\times (0,1)$ is broken into three orbits, covering $\{z\}\times (0,1/2)$, $\{z\}\times \{1/2\}$ and $\{z\}\times (1/2,1)$ respectively (see Figure~\ref{annulus}). We define $\varphi|_{H_i}=\bar{h}_i^{-1}\circ\phi$. 

\begin{figure}[h]
\centering
\includegraphics[scale=0.4]{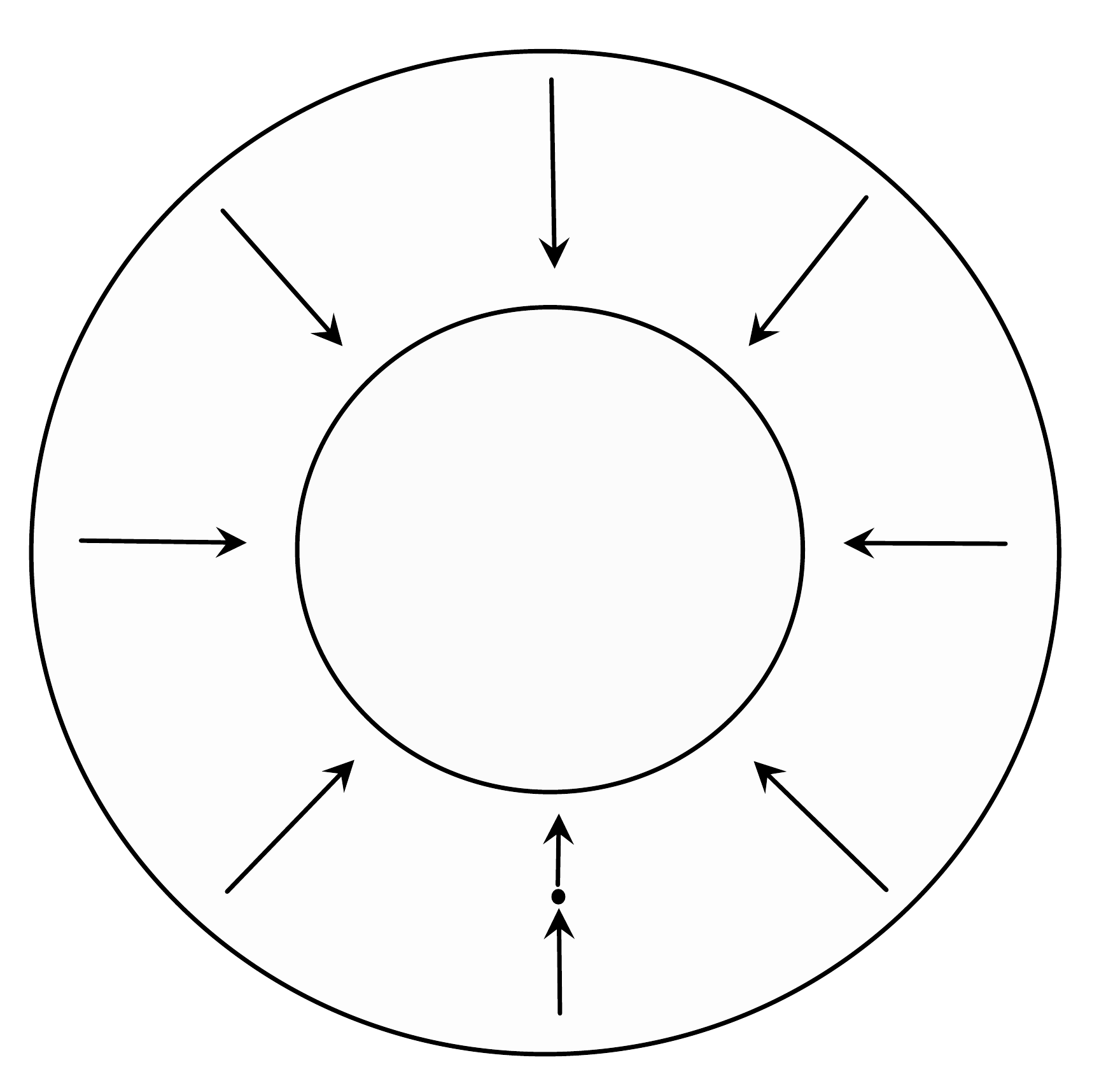}
\caption{Flow on $S^1\times[0,1]$ which has $S^1\times \{0\}$ as a repelling circle of fixed points, $S^1\times\{1\}$ as an attracting circle of fixed points and the point $\{z\}\times\{1/2\}$ as a degenerate saddle fixed point.}
\label{annulus}
\end{figure}

To construct the flow in case (B) consider a sphere $S\subset\mathbb{R}^3$ embedded in such a way that the height function $f:S\longrightarrow\mathbb{R}$ with respect to some plane has $k_i$ maxima, all of them contained in the same level set $f^{-1}(c)$, $k_i-1$ saddle critical points, all of them in the same level set $f^{-1}(b)$ and one minimum at $f^{-1}(a)$ (see Figure~\ref{Morse}). To simplify the notation we shall denote by $S_A$ to $f^{-1}(A)$ for any $A\subset\mathbb{R}$ and by $S_p$ to $f^{-1}(p)$ for any $p\in\mathbb{R}$.

If we choose $\varepsilon>0$ sufficiently small we have that $S_{[a,a+\varepsilon]}\cup S_{[c-\varepsilon,c]}$ is the disjoint union of $k_i+1$ closed topological disks and there exists a homeomorphism $\hat{h}_i:H_i\longrightarrow S\setminus(S_{[a,a+\varepsilon)}\cup S_{(c-\varepsilon,c]})$. We consider a flow $\varphi_{Y}$ in $S$ induced by a vector field $Y$ which satisfies that $-Y$ is gradient-like with respect to $f$. The flow $\varphi_Y$ has $k_i$ repelling hyperbolic fixed points, $k_i-1$ saddle hyperbolic fixed points and one attracting hyperbolic fixed point. If $p$ is one of the saddle fixed points, each one of the two branches of its stable manifold tends to a different repelling point in negative time, while  both two branches of its unstable manifold are attracted by the attracting fixed point.  Notice that the $k_i+1$ disjoint topological circles $S_{(a+\varepsilon)}\cup S_{(c-\varepsilon)}$ form  a local section for the flow $\varphi_Y$. Invoking again Beck's theorem we construct a new flow $\hat{\varphi}$ which is stationary in $S_a\cup S_c$ and such that, each non-stationary trajectory of $\varphi_Y$ is borken into (a) three trajectories if it connects a repelling point with a saddle or a saddle with the attracting point or (b) into four trajectories if it connects a repelling point with the attracting point. A trajectory $\gamma$ in case (a) splits into the oriented ray of $\gamma$ which connects a repelling point (resp. a saddle point) with a point of $S_{c-\varepsilon}$ (resp. with a point of $S_{a+\varepsilon}$), the right endpoint of this ray, and the oriented ray of $\gamma$ that connects this endpoint with the saddle point. A trajectory $\bar{\gamma}$ in case (b) splits into the oriented ray of $\bar{\gamma}$ that connects a repelling point with a point of $S_{c-\varepsilon}$, the right endpoint of this ray, the oriented ray which connects this endpoint with a point of $S_{a+\varepsilon}$, the right endpoint of this ray, and the ray that connects this endpoint with the attracting fixed point. Notice that $\hat{h}_i(H_i)=S\setminus(S_{[a,a+\varepsilon)}\cup S_{(c-\varepsilon,c]})$ is invariant under $\hat{\varphi}$ and, hence, we define $\varphi|_{H_i}=\hat{h}_i^{-1}\circ\hat{\varphi}|_{\hat{h}(H_i)}$. This flow admits a Morse decomposition $\{M_1,M_2,M_3\}$, where $M_1$ is an attracting circle of fixed points, $M_2$ is the union of  the $k_i-1$ topologically hyperbolic saddle fixed points  and $M_3$ is a union of $k_i$ repelling circles of fixed points (see Figure~\ref{handle}).

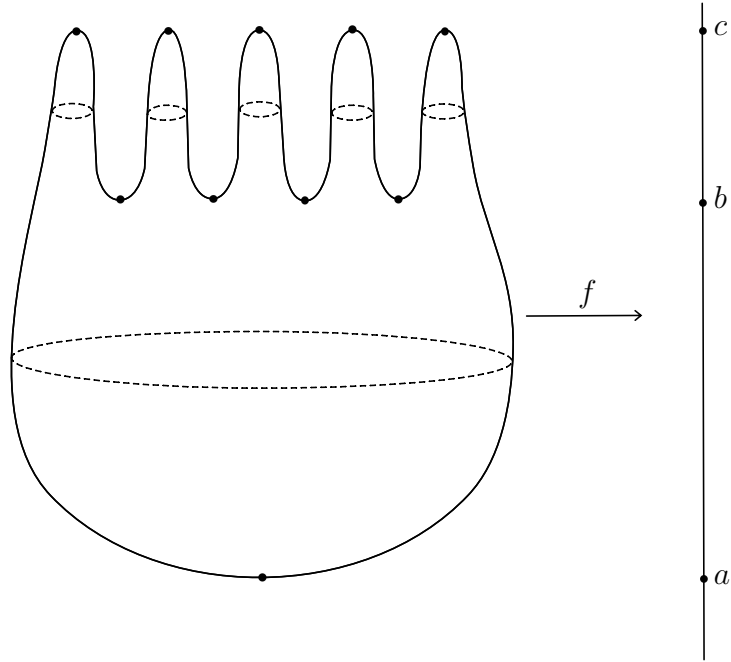
\begin{figure}
\setlength{\unitlength}{1cm}
\begin{picture}(10,9)
\put(0,-.5){\includegraphics[scale=0.3]{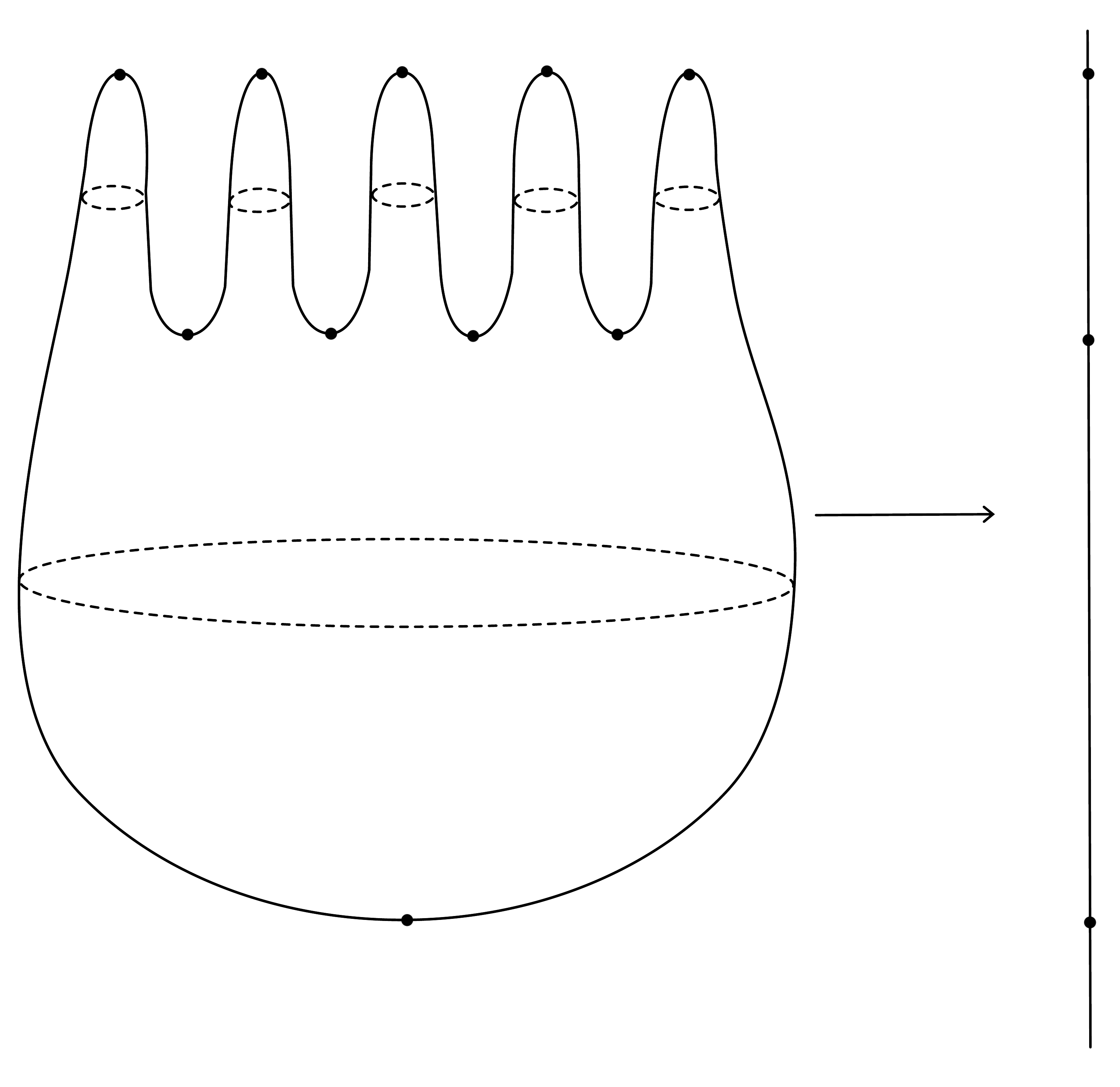}}
\put(7.7,4.6){$f$}
\put(9.5,0.85){$a$}
\put(9.5,5.85){$b$}
\put(9.5,8.15){$c$}
\end{picture}
\caption{A sphere in $\mathbb{R}^3$ embedded in such a way that the height function with respect to some plane has five maxima at height $c$, four saddle critical points at height $b$ and one minimum at height $a$.}
\label{Morse}
\end{figure}

It is clear from the construction that $K$ is an isolated non-saddle set and that if we denote by $L$ the set comprised of the isolated fixed points, $\mathcal{I}(K)=M\setminus L$. In addition, $\mathcal{I}(K)\setminus K$ is the disjoint union of $\hat{D}_1,\ldots, \hat{D}_l,\hat{H}_{l+1},\ldots, \hat{H}_n$, where the symbol $\hat{}$ indicates that we are removing the fixed points. By the construction, each component $\hat{D}_i$ has local complexity $k_i=0$ and each component $\hat{H}_i$ has local complexity $k_i\geq 1$. As a consequence, $\mathcal{I}(K)$ has complexity $g$,  Moreover, each $\hat{H}_i$ contains the sets $W^u(p)\setminus\{p\}$ and $W^s(p)\setminus\{p\}$ for some isolated saddle fixed point $p$, which are dissonant points for $K$. The number of attracting, topologically hyperbolic saddle and degenerate saddle fixed points is clear from the construction.

\begin{figure}[h]
\setlength{\unitlength}{1cm}
\includegraphics[scale=0.3]{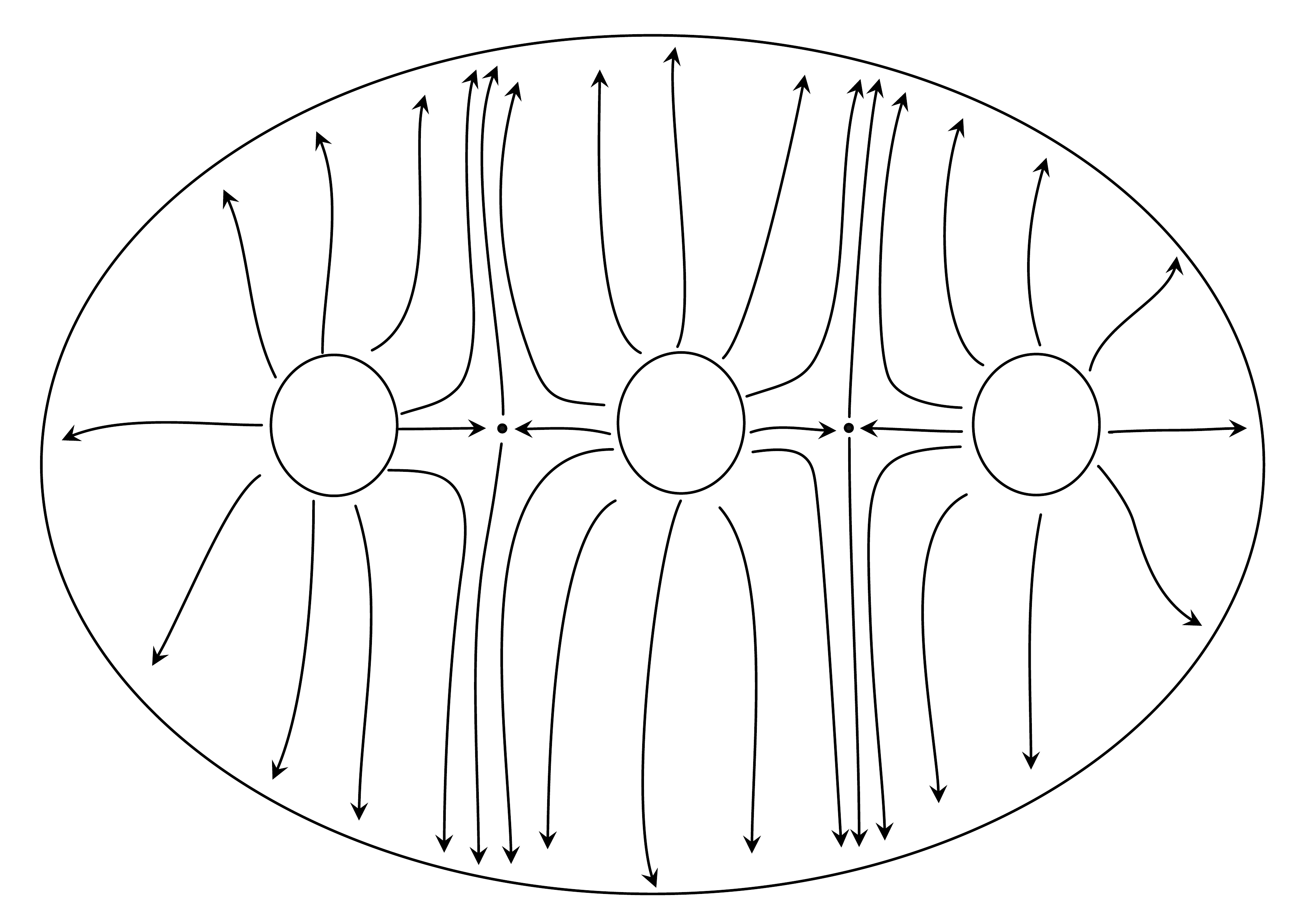}
\caption{Flow defined on a sphere with the interior of four topological closed disks removed. This flow has a Morse decomposition $\{M_1,M_2,M_3\}$ where $M_1$ is the attracting outer circle of fixed points, $M_2$ is the union of two topologically hyperbolic saddle fixed points and $M_3$ is the union of the three repelling inner circles of fixed points.}
\label{handle}
\end{figure}
\end{proof}

\section{Dynamical and homological robustness of non-saddle sets}\label{sec:4}

In this section we study necessary and sufficient conditions for the preservation of the property of being non-saddle by continuations. We start by recalling the basic notions of continuation theory.  

Let $M$ be an $n$-dimensional smooth manifold. We say that the family of flows $\varphi_\lambda:M\times\mathbb{R}\to  M$,  with $\lambda$ in the unit interval $I$, is a \emph{differentiable parametrized family of flows} if  the map $\varphi:M\times\mathbb{R}\times I\to M$ given by $\varphi(x,t,\lambda)=\varphi_\lambda(x,t)$ differentiable. In this context,  the family $(K_\lambda)_{\lambda\in J}$,  where $J\subset[0,1]$ is a closed (non-degenerate) subinterval and, for each $\lambda\in J$, $K_\lambda$ is an isolated invariant set for $\varphi_\lambda$ is said to be a \emph{continuation} if for each $\lambda_0\in J$ and each $N_{\lambda_0}$ isolating neighborhood for $K_{\lambda_0}$, there exists $\delta>0$ such that $N_{\lambda_0}$ is an isolating neighborhood for $K_\lambda$ for every $\lambda\in (\lambda_0 -\delta, \lambda_0 + \delta)\cap J$. We say that the family $(K_\lambda)_{\lambda\in J}$ is a continuation of $K_{\lambda_0}$ for each $\lambda_0\in J$.

Notice that \cite[Lemma~6.1]{Sal} ensures that if $K_{\lambda_0}$ is an isolated invariant set for $\varphi_{\lambda_0}$, there always exists  a continuation $(K_\lambda)_{\lambda\in J_{\lambda_0}}$ of $K_{\lambda_0}$ for some closed (non-degenerate) subinterval $\lambda_0\in J_{\lambda_0}\subset[0,1]$.

There is a simpler definition of continuation based on \cite[Lemma 6.2]{Sal}. There, it is proved that if $\varphi_\lambda : M \times\mathbb{R}\to M$ is a parametrized family of flows and if $N_1$ and $N_2$ are isolating neighborhoods of the same isolated invariant set for $\varphi_{\lambda_0}$, then there exists $\delta>0$ such that $N_1$ and $N_2$ are isolating neighborhoods for $\varphi_\lambda$, for every $\lambda\in(\lambda_0-\delta,\lambda_0 +\delta)\cap[0,1]$, with the property that, for every $\lambda$, the isolated invariant subsets in $N_1$ and $N_2$ which have $N_1$ and $N_2$ as isolating neighborhoods agree. Therefore, the family $(K_\lambda)_{\lambda\in J}$, with $K_\lambda$ an isolated invariant set for $\varphi_\lambda$, is a continuation if for every $\lambda_0\in J$ there are an isolating neighborhood $N_{\lambda_0}$ for $K_{\lambda_0}$ and a $\delta > 0$ such that $N_{\lambda_0}$ is an isolating neighborhood for $K_\lambda$, for every $\lambda\in(\lambda_0-\delta,\lambda_0 +\delta)\cap J$.

It  is well known (see \cite[Theorem~4]{sanjuni}) that attractors are robust from the dynamical and topological points of view since an attractor $K$ locally continues to a family of attractors with the shape of $K$. This is not the case for isolated non-saddle sets as it was shown in \cite{GSRo}  (see Figure~\ref{continuation}). In fact, neither the property of being non-saddle nor the topological properties of the original non-saddle set are preserved by local continuation. However, it turns out that there exist some relations between the preservation of certain topological properties by continuation and the preservation of the dynamical property of non-saddleness. For instance, if the phase space is a surface \cite[Theorem~26]{BNLA} or is a closed orientable smooth manifold with trivial integral first cohomology group \cite[Theorem~39]{BSdis}, the property of being non-saddle is preserved by continuation if and only if the shape is preserved.   

\begin{figure}
\centering
\includegraphics[scale=1]{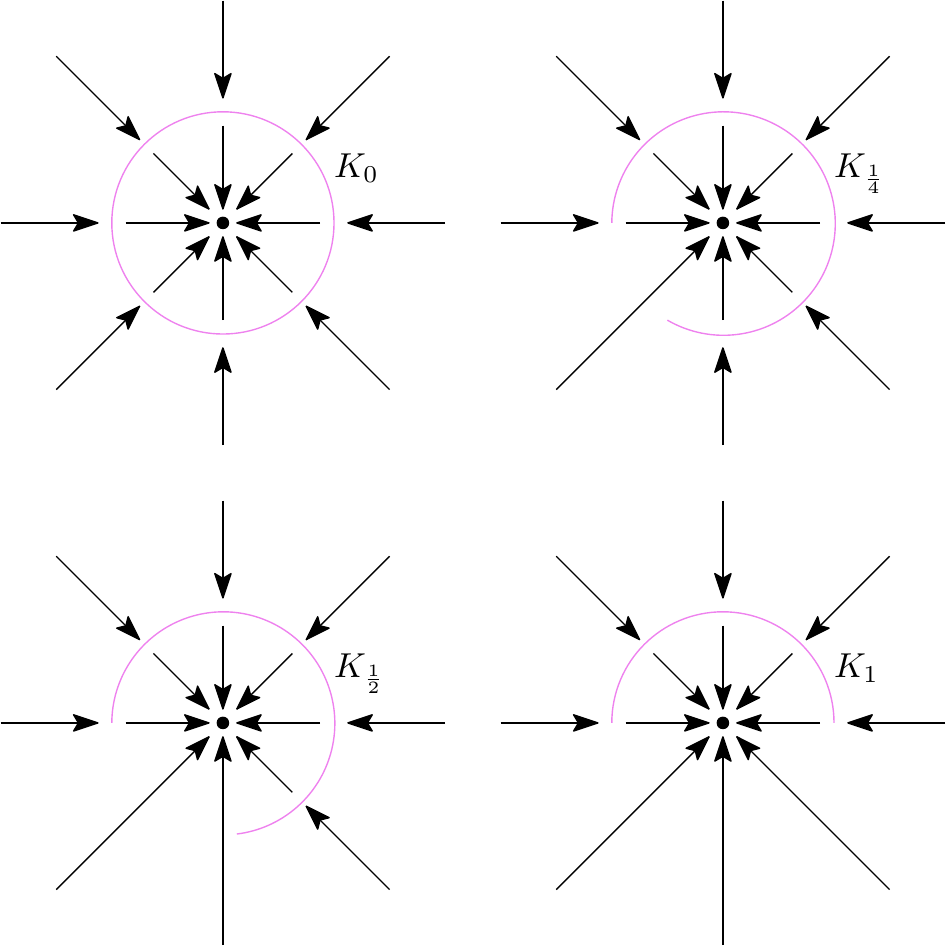}
\caption{An isolated non-saddle circle which continues to a family of saddle sets which are contractible.}
\label{continuation}
\end{figure}

The next result can be extracted from the proof of \cite[Theorem~5]{GSRo} and is a consequence of the robustness properties of transversality.

\begin{lemma}\label{robust}
Let $\varphi_\lambda:M\times\mathbb{R}\to M$ be a smooth parametrized family of flows (parametrized by $\lambda\in I$, the unit interval) defined on an $n$-dimensional smooth manifold. Suppose that $K_0$ is a connected isolated non-saddle set for $\varphi_0$ and that $N=N^+\cup N^-$ is a differentiable isolating block manifold for $K_0$. Then, there exists $\lambda_0>0$ such that $N$ is an isolating block for $0<\lambda<\lambda_0$ with the same entrance and exit sets. 
\end{lemma}

The following proposition gives a necessary and suficient condition for the robustness of non-saddleness.

\begin{proposition}\label{rchar}
Let $\varphi _{\lambda },$ with $\lambda \in \lbrack 0,1]$, be a differentiable parametrized family of flows on a connected differentiable $n$-manifold $M$ and $K_{0}$ be a connected isolated non-saddle set of $\varphi_0$. Suppose that $K_0$ continues to a family $(K_\lambda)_{\lambda\in [0,\delta]}$ of non-empty isolated invariant compacta. Then,  $K_\lambda$ is non-saddle for $\lambda>0$ sufficiently small if and only if there exists a connected differentiable isolating block manifold $N$ of $K_0$ such that $N$ isolates $K_\lambda$  and that each component of $N\setminus K_\lambda$ contains exactly a component of $\partial N$.
\end{proposition}

\begin{proof}
Suppose that $K_\lambda$ is non-saddle for $\lambda$ sufficiently small. Then, if we choose $N$ to be a connected differentiable isolating block manifold of $K_0$, Lemma~\ref{robust} ensures that there exists $\lambda_0>0$ such that, for $\lambda\in[0,\lambda_0)$, $N$ is an isolating block of the form $N=N^+\cup N^-$ for $K_\lambda$. The necessity follows from the fact that $\partial N$ is a deformation retract of $N\setminus K_\lambda$.

Conversely, suppose that there exists a connected differentiable isolating block manifold $N$ of $K_0$ and a $\lambda_1>0$ such that, for every $\lambda\in[0, \lambda_1)$, $N$ isolates $K_\lambda$ and  each component of $N\setminus K_\lambda$ contains exactly a component of $\partial N$. Lemma~\ref{robust} guarantees that there exists $\lambda_1'>0$ such that $N$ is an isolating block of $K_\lambda$ for $\lambda\in[0,\lambda_1')$ satisfying that the entrance and exit sets for $\varphi _{\lambda }$ are disjoint and that they agree with those for $\varphi_{0}.$ We may assume that in fact $\lambda_1'=\lambda_1$. Suppose that  $K_\lambda$ is saddle for some $\lambda\in (0,\lambda_1)$. Then, there exists $x_0\in N\setminus (N^+\cup N^-)$ and, hence, if we consider the entrance and exit times $t_\lambda^i(x_0)$ and $t_\lambda^o(x_0)$ we have that $x_0[t_\lambda^i(x_0),t_\lambda^o(x_0)]$ is a path in $N\setminus K_\lambda$ joining a component of $N^i$ with a component of $N^o$ which, by the previous discussion must be different components of $\partial N$. This contradiction proves the converse statement.
\end{proof}

\begin{corollary}
 Suppose $\varphi_\lambda: M\times\mathbb{R}\to M$ is a differentiable parametrized family of flows defined on a connected $n$-dimensional differentiable manifold $M$. Let $K_0$ be an isolated non-saddle continuum for $\varphi_0$ and $(K_\lambda)_{\lambda\in[0,\delta]}$ a continuation of $K_0$. Suppose that there exists $\lambda_0>0$ such that, for $\lambda\in[0,\lambda_0)$, $K_0\subset K_\lambda$. Then, $K_\lambda$ is non-saddle for $\lambda>0$ sufficiently small.
\end{corollary}

We state the following result without proof since its proof is a small modification of the proof of \cite[Theorem~39]{BSdis} using Theorem~\ref{structure}.

\begin{theorem}\label{robustness}
Let $\varphi _{\lambda },$ with $\lambda \in \lbrack 0,1]$, be a differentiable parametrized family of flows on a connected $G$-orientable differentiable $n$-manifold $M$ with $H^1(M;G)=0$ and $K_{0}$ be a connected isolated non-saddle set for $\varphi_0$ that continues to a family $(K_\lambda)_{\lambda\in[0,\delta]}$ of non-empty isolated invariant compacta. Then, $K_\lambda$ is non-saddle for $\lambda>0$ sufficiently small if and only if $\check{H}^*(K_0;G)\cong\check{H}^*(K_\lambda;G)$.
\end{theorem}

Theorem~\ref{robustness} establishes that the robustness of a topological property, namely \v Cech cohomology, is equivalent to the robustness of the dynamical property of non-saddleness if the phase space is a smooth manifold with trivial first cohomology group. So far we are unable to establish a equivalence between the robustness of  non-saddleness and the robustness of the \v Cech cohomology  for isolated non-saddle continua without further assumptions. However, we can prove the equivalence between the  robustness of non-saddleness with a strong form of the robustness of \v Cech cohomology.

\begin{theorem}\label{strongrob}
 Suppose $\varphi_\lambda: M\times\mathbb{R}\to M$ is a differentiable parametrized family of flows defined on a connected differentiable $n$-dimensional manifold $M$. Let $K_0$ be an isolated non-saddle continuum for $\varphi_0$ and suppose that $K_0$ continues to a family $(K_\lambda)_{\lambda\in[0,\delta]}$ of non-empty isolated invariant compacta. Then, $K_\lambda$ is non-saddle for $\lambda>0$ sufficiently small if and only if there exists a connected isolating block $N$ of $K_0$ such that $N$ isolates $K_\lambda$ and the inclusion $i_\lambda:K_\lambda\hookrightarrow N$ induces isomorphisms in \v Cech cohomology with $\mathbb{Z}_2$ coefficients.
\end{theorem}

\begin{proof}
Suppose that $K_\lambda$ is non-saddle for $\lambda>0$ sufficiently small. Consider a differentiable isolating block manifold $N$ of $K_0$. Then, by   Lemma~\ref{robust}, $N$ is an isolating block for $K_\lambda$ and, since $K_\lambda$ is non-saddle and $N$ is connected, the inclusion $i_\lambda: K\hookrightarrow N$ induces isomorphisms in \v Cech cohomology with $\mathbb{Z}_2$ coefficients as desired. 

Conversely, suppose that there exists a connected isolating block $N'$ of $K_0$ such that $N'$ isolates $K_\lambda$ for every $\lambda$ smaller that $\lambda_0>0$  and that the inclusion $i'_\lambda:K_\lambda\hookrightarrow N'$ induces isomorphisms in \v Cech cohomology with $\mathbb{Z}_2$ coefficients. Let $N$ be a connected differentiable isolating block manifold for $K_0$ contained in $N'$. Notice that $N$ is of the form $N^+\cup N^-$ and, hence, the flow $\varphi_0$ provides a deformation retraction from $N'$ onto $N$. We may assume that $N$ isolates $K_\lambda$ for $\lambda\in[0,\lambda_0)$ since, otherwise, we only have to choose a smaller $\lambda_0$. Since the inclusion $i'_\lambda:K_\lambda\hookrightarrow N'$ is the composition of the inclusions $i_\lambda:K_\lambda\hookrightarrow N$ and $j: N\hookrightarrow N'$ and $i'_\lambda$ and $j$ induce isomorphisms in \v Cech cohomology, it follows that $i_\lambda$ must also induce isomorphisms in \v Cech cohomology. As a consequence, $\check{H}^k(N,K_\lambda;\mathbb{Z}_2)= 0$ for $k\geq 0$ and $\lambda\in [0,\lambda_0)$. Then, Alexander duality ensures that $H_{n-k} (N\setminus K_\lambda,\partial N;\mathbb{Z}_2 ) = 0$ for each $k$ and from the terminal part of the long exact sequence of singular homology of  the pair $(N\setminus K_\lambda , \partial N )$, we deduce that the homomorphism $i_*^\lambda:H_0 (\partial N;\mathbb{Z}_2 )\to H_0 (N\setminus K_\lambda;\mathbb{Z}_2)$, induced by the inclusion $i^\lambda:\partial N\hookrightarrow N\setminus K_\lambda$, is an isomorphism for each $\lambda\in [0,\lambda_0)$.  Therefore, each component of $N\setminus K_\lambda$ contains exactly one component of $\partial N$ and the result follows from Proposition~\ref{rchar}.
\end{proof}

\begin{remark}
Notice that both Theorem~\ref{robustness} and Theorem~\ref{strongrob} can be established in shape theoretical terms. In Theorem~\ref{robustness} the isomorphism between \v Cech cohomologies may be substituted by a shape equivalence and in Theorem~\ref{strongrob} the condition about the inclusion inducing isomorphisms in \v Cech cohomology can be replaced by the condition of the inclusion being a shape equivalence.
\end{remark}

\section*{Acknowledgement}

The author is deeply grateful to Jos\'e M.R. Sanjurjo for his support and inspiration all these years. 

\bibliography{Biblio1}        
\bibliographystyle{plain}  

\end{document}